\documentclass[a4paper,11pt]{article}
\usepackage{fullpage}

\usepackage{graphics}
\usepackage{theorem}
\usepackage{latexsym}
\usepackage{amssymb}

\newtheorem{theorem}{Theorem}[section]
\newtheorem{lemma}[theorem]{Lemma}
\newtheorem{corollary}[theorem]{Corollary}
\newtheorem{proposition}[theorem]{Proposition}

\theorembodyfont{\rm}
\newtheorem{definition}[theorem]{Definition}
\newtheorem{example}[theorem]{Example}

\newenvironment{proof}[1][]{\medskip\noindent {\it Proof{}#1. }}{\hfill$\Box$\par\medskip}

\newcommand{\LY}{{\rm LY}}
\newcommand{\Int}[1]{{\rm Int}\left(#1\right)}
\newcommand{\IR}{{\mathbb R}}
\newcommand{\eps}{\varepsilon}

\newcommand{\CP}{{\mathcal P}}

\begin{document}

\title{Dense chaos for continuous interval maps
\footnotetext{ Nonlinearity, {\bf 18}, 1691-1698, 2005.}}
\author{Sylvie Ruette}

\date{March 8, 2005.}

\maketitle

\begin{abstract}
A continuous map $f$ from a compact interval $I$ into itself is
densely (resp. generically) chaotic if the set of points $(x,y)$ such that
$\limsup_{n\to+\infty}|f^n(x)-f^n(y)|>0$ and $\liminf_{n\to+\infty}
|f^n(x)-f^n(y)|=0$ is dense (resp. residual) in $I\times I$.
We prove that if the interval map $f$ is densely but not generically
chaotic then there is a descending sequence of invariant intervals, each of
which containing a horseshoe for $f^2$. It implies that every densely
chaotic interval map is of type at most $6$ for Sharkovski\u\i's order
(that is, there exists a periodic point of period $6$), and its topological
entropy is at least $\log 2/2$. We show that equalities can be realised.
\end{abstract}

\section{Introduction}

This article deals with the dynamics of interval maps, that is, continuous
maps $f\colon I\to I$ where $I$ is a compact interval in $\IR$.
We first give some notations used in this paper.
An \emph{invariant} set (for the map $f$) is a closed non empty subset $A$ such that 
$f(A)\subset A$. A \emph{transitive} subset (for $f$) is an invariant set $A$ 
such that $f|_A$ is transitive (see, e.g., \cite{DGS} for the definition 
of transitivity). The length of an interval $J$ is denoted by $|J|$.
An interval $J$ is \emph{non degenerate} if $|J|>0$, that is, $J$ is
neither empty nor reduced to a single point. 

\medskip
In \cite{LY} Li and Yorke called \emph{chaotic} some kind of behaviour
of interval maps, although without formal definition.
The following notions of \emph{Li-Yorke pairs} and \emph{Li-Yorke chaos}
were derived from this article.

\begin{definition}
Let $T\colon X\to X$ be a continuous map on a metric space
$X$.  If $x,y\in X$ and $\delta>0$,  $(x,y)$ is called a {\em Li-Yorke
pair of modulus $\delta$} if
$$
\limsup_{n\to+\infty}d(T^n(x),T^n(y))> \delta\quad\mbox{and}\quad
\liminf_{n\to+\infty}d(T^n(x),T^n(y))=0;
$$
$(x,y)$ is a {\em Li-Yorke pair} if it is a Li-Yorke pair of modulus
$\delta$ for some $\delta>0$.
The set of Li-Yorke pairs of modulus $\delta$ is denoted by
$\LY(T,\delta)$ and the set of Li-Yorke pairs by $\LY(T)$.
\end{definition}

\begin{definition}
Let $T\colon X\to X$ be a continuous map on a metric space
$X$. The system $(X,T)$ is said \emph{chaotic in the sense of Li-Yorke}
if there exists an uncountable set $S\subset X$ such that for all
$x,y\in S, x\not=y$, $(x,y)$ is a Li-Yorke pair.
\end{definition}

The definition of {\em generic chaos} is due to Lasota (see
\cite{Pio}). 
Being inspired by this
definition, Snoha defined {\em generic $\delta$-chaos}, {\em dense
chaos} and {\em dense $\delta$-chaos} \cite{Sno}.

\begin{definition}
Let $T\colon X\to X$ be a continuous map on a metric space
$X$ and $\delta>0$.
\begin{itemize}
\item $T$ is {\em generically chaotic} if $\LY(T)$ is residual in
$X^2$,
\item $T$ is {\em generically $\delta$-chaotic} if $\LY(T,\delta)$  is
residual in $X^2$,
\item $T$ is {\em densely chaotic} if $\LY(T)$ is dense in $X^2$,
\item $T$ is {\em densely $\delta$-chaotic} if $\LY(T,\delta)$  is
dense in $X^2$.
\end{itemize}
\end{definition}

Generic $\delta$-chaos obviously implies both generic chaos and
dense $\delta$-chaos, which in turn imply dense chaos.

In \cite{Sno}, Snoha proved that for an interval map generic chaos
implies generic $\delta$-chaos for some $\delta>0$ and the notions of
generic $\delta$-chaos and dense $\delta$-chaos coincide, but a
densely chaotic interval map may be not generically chaotic.
In \cite{Sno2} Snoha gave a characterisation of densely chaotic
interval maps and proved that for piecewise monotone interval maps
the notions of dense chaos and generic chaos coincide. He 
asked what the infimum of topological entropy and the type for
Sharkovski\u\i's order are for densely chaotic interval maps. 
We recall Sharkovski\u\i's theorem (see \cite{Ste}):

\begin{theorem}[Sharkovski\u\i \cite{Sha}]\label{theo:Sharkovskii}
Consider the following order:
$$
3\lhd 5\lhd 7\lhd 9\lhd\cdots \lhd 2\cdot 3 \lhd 2\cdot 5\lhd
2\cdot 7\lhd\cdots \lhd 2^2\cdot 3\lhd  2^3\cdot 5\lhd\cdots\lhd
2^{n}\lhd\cdots 2^3\lhd2^2\lhd2\lhd 1
$$ 
If the interval map $f$ has a periodic point of period $n$ then
it has periodic points of period $m$ for all $m\rhd n$.
\end{theorem}

According to Theorem~\ref{theo:Sharkovskii}, the set of periods of
periodic points of an interval map $f$ is, either $\{m\mid m\unrhd n \}$ for
some positive integer $n$, and in this case $f$ is said of \emph{type $n$},
or $\{2^n\mid n\geq 0\}$, and $f$ is said of \emph{type $2^{\infty}$}.
Remark that there exist interval maps of all types \cite{Ste, BP}.

\medskip
Our first motivation was to answer Snoha's questions. 
In Theorem~\ref{theo:structure-dense-chaos}, we actually obtain
a result on the structure of interval maps that are densely chaotic 
but not generically chaotic: for such a map $f$, there exists a descending
sequence of invariant intervals, with lengths tending to $0$, and
each of them contains a horseshoe for the map $f^2$ (that is, two closed 
non degenerate intervals $J,K$ with disjoint interiors such that 
$f^2(J)\cap f^2(K)\supset J\cup K$). On the other hand,
Snoha gave a characterisation of generic chaos in term of
transitive subintervals:

\begin{theorem}[Snoha \cite{Sno}]\label{theo:Snoha}
Let $f$ be an interval map.
The following conditions are equivalent:
\begin{itemize}
\item $f$ is generically chaotic,
\item either there exists a unique transitive non degenerate 
subinterval, or there exist two transitive non degenerate
subintervals with a common endpoint. Moreover
for every non degenerate interval $J$, $f^n(J)$ is non degenerate and
there exist a transitive subinterval $T$ and an integer $n\geq 0$
such that $f^n(J)\cap\Int{T}\not=\emptyset$.
\end{itemize}
\end{theorem}

We deduce from the structure of densely chaotic interval maps that
such maps are at most of type $6$ for Sharkovski\u\i's order and their
topological entropy is greater than or equal to $\frac{\log 2}{2}$ 
(Corollary~\ref{cor:dense-chaos-entropy-type}). Example~\ref{ex:equalities}
shows that equalities are possible, and in addition they can be realised by
generically chaotic interval maps.

\medskip
In \cite{Mur} Murinov\'a generalised Snoha's work and 
showed that for a complete metric space $X$,
generic $\delta$-chaos and dense $\delta$-chaos are equivalent.
She also exhibited a generically chaotic system which is
not generically $\delta$-chaotic for any $\delta>0$.

\medskip
If $X$ is a complete metric space and $G\subset X\times X$ is a dense 
$G_{\delta}$-set then using Kuratowski's theorem (see, e.g.,\cite{Oxt})
one can find an uncountable set $S$ such that $S\times S$ deprived of 
the diagonal of $X\times X$ is included in $G$ (see e.g., 
\cite[Lemma~3.1]{HY}). Therefore a generically chaotic map on
a complete metric space is chaotic in the sense of Li-Yorke.
In \cite{KuS}, Kuchta and Sm\'{\i}tal showed that on the interval
the existence of one Li-Yorke pair is enough to imply
chaos in the sense of Li-Yorke, consequently
dense chaos implies Li-Yorke chaos for interval maps.
However it is not known whether dense chaos implies
Li-Yorke chaos in general.

\medskip
Section~\ref{sec:horseshoe} contains our main results.
Some preliminary lemmas are needed, they are stated in 
Section~\ref{sec:preliminary}

\section{Preliminary results}\label{sec:preliminary}

\begin{lemma}\label{lem:dense-chaos}
Let $f$ be a densely chaotic interval map.
\begin{enumerate}
\item
If $J$ is a non degenerate interval then $f^n(J)$ is non degenerate
for all $n\geq 0$.
\item 
Consider disjoint non degenerate intervals $J_1,\ldots,J_p$ such
that $f(J_i)\subset J_{i+1\bmod p}$. Then either $p=1$, or $p=2$ and
$J_1, J_2$ have a common endpoint. If the intervals $J_i$ are closed
then $p=1$.
\item If $J, J'$ are invariant non degenerate intervals,
then $J\cap J'\not=\emptyset$.
\end{enumerate}
\end{lemma}

\begin{proof}

\noindent
i) If $J$ is a non degenerate interval then there exists $(x,y)\in
J\times J$ such that $(x,y)$ is a Li-Yorke pair, thus
$\limsup_{n\to+\infty} |f^n(J)|>0$ and  for every $n\geq 0$ the interval
$f^n(J)$ is not reduced to a point.

\medskip\noindent ii)
Let $J_1,\ldots,J_p$ be disjoint non degenerate intervals such that
$f(J_i)\subset J_{i+1\bmod p}$. Suppose that there exist  $0\leq
i,j\leq p$ such that the distance $D$ between $J_i$ and $J_j$ is
positive. By continuity there exists $\eta>0$ such that if
$|x-y|<\eta$ then $|f^k(x)-f^k(y)|<D$ for all $0\leq k\leq p$.  If
$(x,y)\in J_i\times J_j$ then for all $l\geq 0$ one has
$|f^{lp}(x)-f^{lp}(y)|\geq D$ thus for all $n\geq 0$
$|f^n(x)-f^n(y)|\geq  \eta$, which contradicts the assumption that $f$
is densely chaotic.  If the intervals $J_i$ are closed it implies that
$p=1$; otherwise it implies that $p=1$ or $p=2$, and if $p=2$ then
$J_1$ and $J_2$ have a common endpoint.

\medskip\noindent iii)
Let $J,J'$ be two invariant non degenerate intervals. 
Then there exists a Li-Yorke
pair $(x,x')$ in $J\times J'$, in particular there exists an
increasing sequence $(n_i)$ such that
$\lim_{i\to+\infty}|f^{n_i}(x)-f^{n_i}(x')|=0$. By compactness there
exist $(m_i)$ a subsequence of $(n_i)$ and a point $z$ such that
$\lim_{i\to+\infty}f^{m_i}(x)=\lim_{i\to+\infty}f^{m_i}(x')=z$
and the point $z$ belongs to $J\cap J'$.
\end{proof}

\begin{lemma}\label{lem:nested-Jn-to-0}
Let $f$ be a densely chaotic interval map. Suppose that there
exists a sequence of non degenerate invariant intervals 
$(J_n)_{n\geq 0}$ such that $\lim_{n\to+\infty}|J_n|
=0$. Then there exists a point $z\in \bigcap_{n\geq 0} J_n$ and 
$f(z)=z$. 

Moreover there exists a subsequence of closed non degenerate intervals
$(J'_n)_{n\geq 0}$ such that $f(J'_n)\subset J'_n$, 
$\lim_{n\to+\infty}|J'_n|=0$ and for all $n\geq 0$, $J'_{n+1}$ is included 
in the interior of $J'_n$ with respect to the induced topology on $J'_0$.
\end{lemma}

\begin{proof}
First we show that
$\bigcap_{n=0}^{+\infty} J_n\not=\emptyset$. If $\bigcap_{n=0}^N J_n$
is non degenerate for all $N\geq 0$ then $\bigcap_{n=0}^{+\infty} J_n$
is not empty.  Otherwise let $N\geq 0$ be the greatest integer such
that $\bigcap_{n=0}^N J_n$ is non degenerate.  The interval
$K=\bigcap_{n=0}^N J_n$ is closed, non degenerate and $f(K) \subset
K$. By Lemma~\ref{lem:dense-chaos}(iii), $J_{N+1}\cap K\not=
\emptyset$ thus the set $J_{N+1}\cap K$ is reduced to one point $z$. For
every $n\geq 0$ one has $J_n\cap K\not=\emptyset$ and $J_n\cap
J_{N+1}\not= \emptyset$ by Lemma~\ref{lem:dense-chaos}(iii) thus by
connectedness $z\in J_n$. Consequently $z\in \bigcap_{n=0}^{+\infty}J_n$.

The set $\bigcap_{n=0}^{+\infty}J_n$ is reduced to a single point $z$
because $|J_n|\to 0$. One has $f(z)=z$ because
$f(J_n)\subset J_n$ for all $n\geq 0$.

There exists an increasing sequence $(n_i)_{i\geq 0}$ such that
either $J_{n_i}\cap (z,+\infty)\not=\emptyset$ for all $i\geq 0$ or
$J_{n_i}\cap (-\infty,
z)\not=\emptyset$ for all $i\geq 0$.  
Define $K_n= \bigcap_{i=0}^n J_{n_i}$; this is a closed non degenerate
interval, $f(K_n)\subset K_n$ and $K_{n+1}\subset K_n$.

\medskip\noindent
Case 1. There exists an increasing sequence $(m_i)_{i\geq
0}$ such that $K_{m_{i+1}}\subset \Int{K_{m_i}}$ for all $i\geq 0$.
Take then $J'_i=K_{m_i}$.

\medskip\noindent
Case 2. If the assumption of case 1 is not satisfied then
there exists $N\geq 0$ such that for all $n\geq N$, $K_n\not\subset
\Int{K_N}$, that is, either $\min K_n=\min K_N$ for all $n\geq N$ or
$\max K_n=\max K_N$ for all $n\geq N$. Since $|K_n|\to 0$ one can find
an increasing sequence $(m_i)_{i\geq 0}$ with $m_0=N$ such that
$|K_{m_{i+1}}|<|K_{m_i}|$ for all $i\geq 0$. In this case
$J'_i=K_{m_i}$ is a suitable subsequence of intervals.
\end{proof}

\begin{lemma}\label{lem:interval-no-fixed-point}
Let $J$ be a bounded interval and $f\colon J\to J$ a
continuous map with no fixed point. If $K\subset J$ is a
compact interval then $\displaystyle\lim_{n\to+\infty}|f^n(K)|=0$.
\end{lemma}

\begin{proof}
Suppose that $f(x)<x$ for all $x\in J$, the case $f(x)>x$ for all
$x\in J$ being similar. Write $f^n(K)=[a_n,b_n]$. For every $n\geq 0$
there exists $x\in [a_n,b_n]$ such that $f(x)=b_{n+1}$, thus
$b_{n+1}<x\leq b_n$. The sequence $(b_n)_{n\geq 0}$ is decreasing thus
has a limit in $\overline{J}$; let $z=\lim_{n\to+\infty} b_n$. 
Suppose that $z\in J$. Let $\eps>0$ such that $f(z)+\eps<z$.
By continuity there exists $\eta>0$ such that if $|x-z|<\eta$
then $|f(x)-f(z)|<\eps$. Let $n\geq 0$ such that $|b_n-z|<\eta$.
Then for all $x\in [z,b_n]$ one has $f(x)<f(z)+\eps<z$ and for
all $x\in [a_n,z)$ one has $f(x)<x<z$. This implies that $b_{n+1}<z$,
which is absurd.
Hence $z=\inf J$ and $z<a_n\leq b_n$
for all $n\geq 0$. This implies that $|f^n(K)|\leq |b_n-z|\to 0$ when
$n$ goes to infinity.
\end{proof}

\begin{proposition}[Snoha \cite{Sno}]\label{prop:Snoha}
Let $f$ be an interval map and $\delta>0$.
The following conditions are equivalent:
\begin{itemize}
\item $f$ is generically $\delta$-chaotic,
\item for all non degenerate intervals
$J_1,J_2$ one has $\limsup_{n\to+\infty}|f^n(J_1)|>\delta$ and
$\liminf_{n\to+\infty}{\rm dist}(f^n(J_1),f^n(J_2))=0$ (where
${\rm dist}(\cdot,\cdot)$ denotes the distance between two sets),
\end{itemize}
\end{proposition}

\begin{lemma}\label{lem:dense-chaos-dense-delta-chaos}
Let $f$ be a densely chaotic interval map.
Suppose that there exists $\eps>0$ such that, for every non degenerate
invariant interval $J$, $|J|\geq\eps$. Then $f$ is generically chaotic.
\end{lemma}

\begin{proof}
Suppose that
\begin{equation}\label{eq:fnJ<delta}
\forall \delta>0,\ \exists\, J \mbox{ closed non degenerate interval, } \forall n\geq 0,\ |f^n(J)|\leq\delta.
\end{equation}
We are going to show that is not possible. Let $0<\delta<\eps/4$
and let $J$ be a
closed non degenerate interval such that $|f^n(J)|\leq\delta$ for all
$n\geq 0$. There exists a Li-Yorke pair $(x,y)\in J\times J$ because $f$ is
densely chaotic, thus
\begin{equation}\label{eq:limsup-fnJ}
\limsup_{n\to+\infty}|f^n(J)|>0.
\end{equation}
This implies that there exist $N,p$ such that $f^N(J)\cap
f^{N+p}(J)\not= \emptyset$, thus $f^n(J)\cap f^{n+p}(J)\not=\emptyset$ for 
all $n\geq N$. Since $f^n(J)$ is an interval, this implies that, for every
$0\leq i\leq p-1$, the set $Z_i=\bigcup_{k\geq 0}f^{N+i+kp}(J)$ is an
interval, too. Consequently, the set
$Z=\bigcup_{n\geq N} f^n(J)$ has
at most $p$ connected components, which are non degenerate by
Lemma~\ref{lem:dense-chaos}(i). The image of a connected component is
connected and $f(Z_i)\subset Z_{i+1\bmod p}$, thus the connected components
of $Z$ are necessarily cyclically mapped into each other and
Lemma~\ref{lem:dense-chaos}(ii) applies: 
$Z$ has, either one connected component,
or two connected components with a common endpoint, and $\overline{Z}$
is a closed interval.

If there exist a point $z$ and an integer $n_0\geq N$ such that
$f^2(z)=z$ and $z\in f^{n_0}(J)$ then $z\in f^{n_0+2k}(J)$ for all
$k\geq 0$. Since $|f^n(J)|\leq\delta$ for all $n\geq 0$ one gets that
$\left|\bigcup_{k\geq 0}f^{n_0+2k}(J)\right|\leq 2\delta$ and
$\left|\bigcup_{k\geq 0}f^{n_0+2k+1}(J)\right|\leq 2\delta$.  Let
$L=\overline{\bigcup_{n\geq n_0}f^n(J)}=f^{n_0-N}(\overline{Z})$.
Then $L$ is a closed non degenerate interval, $f(L)\subset L$ and
$|L|\leq 4\delta$. Moreover $|L|\geq\eps$ according to the hypothesis 
of the Lemma,
which is a contradiction because we have
chosen $\delta<\eps/4$. We deduce that $Z$ contains no point $z$ such
that $f^2(z)=z$.

Let $Z_0$ be the connected component of $Z$ containing $f^N(J)$ and
put $g=f^2$. Then $g(Z_0)\subset Z_0$ and $g|_{Z_0}$ has no fixed
point.  The interval $K=f^N(J)$ is compact because $J$ is 
compact and $f^N$ is continuous, so 
Lemma~\ref{lem:interval-no-fixed-point} applies 
and we get that $\lim_{n\to+\infty}|f^{N+2n}(J)|=0$. By
continuity of $f$ we get that $|f^{n}(J)|\to 0$ when $n$ goes to
infinity, which contradicts Equation~(\ref{eq:limsup-fnJ}). We
conclude that Equation~(\ref{eq:fnJ<delta}) is false, consequently
there exists $\delta>0$ such that for all closed non degenerate intervals $J$,
$\limsup_{n\to+\infty}|f^n(J)|\geq\delta$.
The map $f$ is densely chaotic thus for every non degenerate intervals
$J_1,J_2$ there is a Li-Yorke pair in $J_1\times J_2$, hence
$\liminf_{n\to+\infty} {\rm dist}(f^n(J_1),f^n(J_2))=0$. Then
Proposition~\ref{prop:Snoha} implies that $f$ is generically 
chaotic.
\end{proof}

\section{Structure of densely chaotic interval maps}\label{sec:horseshoe}

Recall that the interval map $g$ has a horseshoe if
there exist two closed non degenerate subintervals $J,K$ with
disjoint interiors such that $g(J)\cap g(K)\supset J\cup K$.

\begin{theorem}\label{theo:structure-dense-chaos}
Let $f$ be an interval map. If $f$ is densely chaotic but not generically
chaotic then 
there exists a sequence of invariant non degenerate subintervals
$(J_n)_{n\geq 0}$ such that $J_{n+1}\subset J_n$, $\lim_{n\to+\infty}|J_n|=0$,
and $f^2|_{J_n}$ has a horseshoe for all $n\geq 0$.
\end{theorem}

We need two lemmas in the proof of this theorem.
Lemma~\ref{lem:U<D} is proven in \cite{LMPY} under slightly weaker hypotheses,
see also \cite[p 28]{BCop}. Lemma~\ref{lem:alternating} can be found
in \cite[p31]{BCop}.

\begin{lemma}\label{lem:U<D}
Let $f$ be an interval map with no horseshoe and
$x$ a point .  Write $x_n=f^n(x)$ for all $n\geq 0$. Suppose that
$x_{n+1}\geq x_n$ and $x_{m+1}\leq x_m$. Then $x_n\leq x_m$.
\end{lemma}

\begin{lemma}\label{lem:alternating}
Let $f$ be an interval map such that $f^2$ has no
horseshoe.  Let $x$ be a point which is not ultimately periodic
and write $x_n=f^n(x)$ for $n\geq 0$. Suppose that  there exists
$k_0\geq 2$ such that either $x_{k_0}<x_0<x_1$ or
$x_{k_0}>x_0>x_1$. Then there exist a fixed point $z$ and an integer
$N$ such that, for all $n\geq N$, $x_n>z\Leftrightarrow x_{n+1}<z$.
\end{lemma}

\begin{proof}[ of Theorem~\ref{theo:structure-dense-chaos}]
By assumption the map $f$ is not generically chaotic thus, by 
Lemma~\ref{lem:dense-chaos-dense-delta-chaos}, 
$$ 
\forall \eps>0,\mbox{ there exists an invariant non degenerate interval } J
\mbox{ such that  }|J| <\eps.
$$ 
Let $(I_n)_{n\geq 0}$ be a sequence
of invariant non degenerate intervals $I_n$ such that $|I_n|\to 0$.
By Lemma~\ref{lem:nested-Jn-to-0} there exists 
a sequence of invariant non
degenerate intervals $(J_n)_{n\geq 0}$ such that 
$\lim_{n\to+\infty}|J_n|=0$, and $J_{n+1}\subset \Int{J_n}$ with respect
to the induced topology on $J_0$. From now on we fix $n_0\geq 0$ and
we restrict to
the interval $J_{n_0}$. The map $f|_{J_{n_0}}$ is densely chaotic,
the set $\bigcap_{n\geq n_0} J_n$ is reduced to a single point $z$ and
$f(z)=z$.
Let 
$$\CP=\{x\in J_{n_0}\mid \exists p\geq 1,\lim_{n\to+\infty}f^{np}(x)
\mbox{ exists}\}.
$$
If $x,y\in \CP$ then $(x,y)$ is not a Li-Yorke pair, thus the set 
$J_{n_0}\setminus \CP$ is not empty. 

Assume that $f^2|_{J_{n_0}}$ 
has no horseshoe; we are going to prove that this is
absurd. Let $x_0\in J_{n_0}\setminus \CP$ and write
$x_n=f^n(x_0)$ for all $n\geq 0$. According to Lemma~\ref{lem:alternating}
there exist a fixed point $c$ and an integer $N$ such that, for all $n\geq
0$, $x_{N+2n}<c<x_{N+2n+1}$. Suppose for instance that $c\leq z$, the
case with reverse inequality being symmetric. Since $x_0\not\in\CP$
the sequence $(x_{N+2n})_{n\geq 0}$ is not ultimately monotone, thus
there exists $i\geq 0$ such that 
$$
x_{N+2i+2}<x_{N+2i}<c\leq z.
$$
By continuity there exists a closed non degenerate interval $K$ containing
$x_{N+2i}$ such that $z\not\in K$ and for all $y\in K$, $f^2(y)<y$.
Let $k\geq n_0$ such that $K<J_k$.

The set $K\times K$ contains a Li-Yorke pair because $f$ is densely
chaotic, thus $\limsup_{n\to+\infty}
|f^n(K)|>0$ and there exist $p,q$ such that $f^{q+p}(K)\cap f^q(K)
\not=\emptyset$. Let $L=\overline{\bigcup_{n\geq q}f^n(K)}$. 
One has $f(L)\subset L$. The same argument as for $\overline{Z}$ in
the proof of Lemma~\ref{lem:dense-chaos-dense-delta-chaos} implies that
$L$ is an invariant non degenerate interval.
Moreover Lemma~\ref{lem:dense-chaos}(iii) implies that
$L\cap J_n\not=\emptyset$ for all $n\geq n_0$. Since $J_{k+1}\subset
\Int{J_k}$, this implies that there exists an integer $n \geq 0$ such that
$f^n(K)\cap \Int{J_k}\not=\emptyset$, thus
there exists a closed non degenerate subinterval $K'\subset K$ such that
$f^n(K')\subset J_k$.

Let $m_0\geq n/2$ and $g=f^2$. For all $y\in K'$ and
all $m\geq m_0$ one has $g^m(y)\in J_k$ because $f(J_k)\subset J_k$,
thus
$$
g(y)<y<g^m(y).
$$
This implies that there exists $0<j<m_0$ such that $g^j(y)<g^{j+1}(y)$.
By assumption $g$ has no horseshoe thus $g^j(y)\leq y$ by
Lemma~\ref{lem:U<D}. For all $m\geq m_0$, one has $y\leq g^m(y)$,
thus the same lemma implies that $g^{m+1}(y)\leq g^m(y)$. Consequently,
$(g^m(y))_{m\geq m_0}$
is a non increasing sequence, thus it converges. But this implies that
$K'\times K'$ contains no Li-Yorke pair, which contradicts the fact that
$f$ is densely chaotic. This concludes the proof.
\end{proof}


Next theorem sums up two results on horseshoes, the first point is due
to Block and Coppel \cite{BCop2}, the second one derives from 
\cite{BGMY} (see also \cite[p 196]{BCop}).

\begin{theorem}\label{theo:horseshoe-entropy-type3}
Let $f$ be an interval map with a horseshoe. Then
\begin{itemize}
\item $f$ is of type $3$ for Sharkovski\u\i's order,
\item $h_{top}(f)\geq\log 2$.
\end{itemize}
\end{theorem}

According to Theorem~\ref{theo:Snoha}, a generically chaotic interval map
$f$ admits a transitive subinterval, thus next theorem implies that
$f^2$ has a horseshoe.

\begin{theorem}[Block-Coven \cite{BCov}]\label{theo:transitivity-horseshoe}
Let $f$ be a transitive interval map. Then $f^2$ has a
horseshoe. 
\end{theorem}



\begin{corollary}\label{cor:dense-chaos-entropy-type}
Let $f$ be a densely chaotic interval map. Then $f^2$ has a horseshoe,
$h_{top}(f)\geq \frac{\log 2}{2}$ and $f$ is of type at most $6$ for 
Sharkovski\u\i's order.
\end{corollary}

\begin{proof}
If $f$ is generically chaotic then $f^2$ has a horseshoe by
Theorems \ref{theo:Snoha} and \ref{theo:transitivity-horseshoe}, otherwise
$f^2$ has a horseshoe by Theorem~\ref{theo:structure-dense-chaos}.
By Theorem~\ref{theo:horseshoe-entropy-type3},
$h_{top}(f)\geq \frac{\log 2}{2}$ and $f^2$ has a periodic point of period
$3$, thus $f$ has a periodic point of period $3$ or $6$.
\end{proof}

\begin{example}\label{ex:equalities}
In Corollary~\ref{cor:dense-chaos-entropy-type} equalities are possible.
Consider the ``square-root'' of the tent map, pictured on
Figure~\ref{fig:htop-transitive}. 
The map $g$ swaps the intervals $[0,1/2]$ and $[1/2,1]$ 
thus every periodic point $x\not=1/2$ has an even period.
Moreover the intervals $[0,1/4], [1/4,1/2]$ form a horseshoe for $g^2$, so
Theorem~\ref{theo:horseshoe-entropy-type3} implies that
$g$ is of type $6$ for Sharkovski\u\i's order. 
The map $g^2$ restricted to either $[0,1/2]$ or $[1/2,1]$ is the classical
tent map (upside down on $[0,1/2]$), which is known to be mixing
(see, e.g., \cite[p 159]{BCop}), so $g$ is transitive (thus
generically chaotic).
Finally the topological entropy of $g^2$ is equal to $\log 2$
(use either the fact that it is Markov or the combination of
Theorem~\ref{theo:horseshoe-entropy-type3} and 
\cite[Proposition~(14.20)]{DGS}), hence $h_{top}(g)=\frac{\log 2}{2}$.

\begin{figure}[htb]
\centerline{\includegraphics{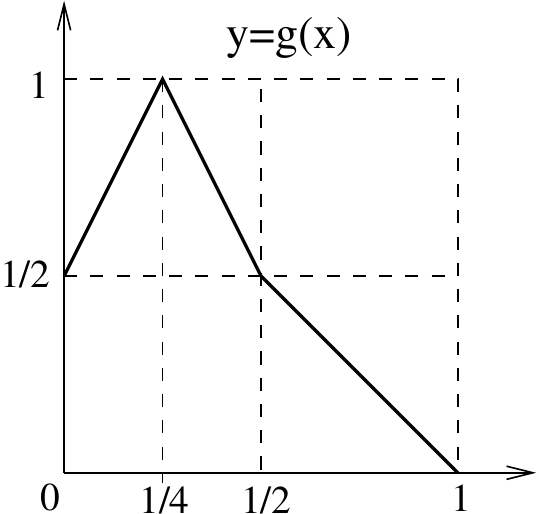}}
\caption{Densely chaotic map of entropy $\frac{\log 2}{2}$ and type
$6$.
\label{fig:htop-transitive}}
\end{figure}
\end{example}

This example shows that the infimum of the topological entropy of
densely (respectively generically) chaotic interval maps is reached 
and is equal to $\frac{\log 2}{2}$.

There also exist transitive (thus generically chaotic) interval maps
of type $2k+1$ for all $k\geq 1$ \cite{BCov}. It derives from
\cite{BGMY} that the topological entropy of a map of type $2k+1$ is greater
than $\frac{\log 2}{2}$.

\section*{References}
\providecommand{\MR}{\relax\ifhmode\unskip\space\fi MR }

\noindent
Laboratoire de Math\'ematiques -- Topologie et Dynamique -- B\^atiment 425 -- Universit\'e Paris-Sud~-- 91405 Orsay cedex -- France\\
{\it e-mail :} {\tt  sylvie.ruette@math.u-psud.fr}
\end{document}